\documentclass[a4,12pt]{article}
\usepackage{graphicx}   
\usepackage{amsmath,amsfonts,amssymb,amsthm}
\usepackage[margin=2.5cm]{geometry}
\usepackage{txfonts}%
\usepackage{bm}

\theoremstyle{plain}
\newtheorem{theorem}{Theorem}[section]
\newtheorem{lemma}[theorem]{Lemma}

\theoremstyle{definition}

\theoremstyle{remark}

\begin{document}

\title{Exponential decay rate of partial autocorrelation coefficients of ARMA and short-memory processes}
\author{
Akimichi Takemura\thanks{Graduate School of Information Science and Technology, University of Tokyo}
}
\date{January, 2016}
\maketitle

\begin{abstract}
We present a short proof of the fact that the exponential decay rate of partial autocorrelation coefficients of a short-memory process, in particular an ARMA process, 
is equal to the exponential decay rate of the coefficients of its infinite autoregressive representation.  
\end{abstract}


\section{Introduction}
The autocorrelation coefficients and the partial autocorrelation coefficients are basic tools for model selection in time series analysis based on ARMA models.  
For AR models, by the Yule-Walker equation, the autocorrelation coefficients satisfy a linear
difference equation with constant coefficients and hence the autocorrelation coefficients decay to zero exponentially with the rate of the reciprocal of the smallest absolute value of the roots of the characteristic polynomial of the AR model.  This also holds for ARMA models, because their
autocorrelations satisfy the same difference equation defined by their AR part, except for some
initial values. 

On the other hand, it seems that no clear statement and proof is given
in standard textbooks on time series analysis 
concerning the exponential decay rate of the partial autocorrelation coefficients for MA models and ARMA models. 
For example, in Section 3.4.2 of \cite{box-jenkins-5th} the following is stated  without clear indication of the decay rate.
\begin{quote}
Hence, the partial autocorrelation function of a mixed process is infinite in extent. It behaves eventually like the partial autocorrelation function of a pure moving average process, being dominated by a mixture of damped exponentials and/or damped sine waves, depending on the order of the moving average and the values of the parameters it contains.
\end{quote}
In Section 3.4 of \cite{brockwell-davis} the following is stated on MA($q$) processes
again without clear indication of the decay rate.
\begin{quote}
In contrast with the partial autocorrelation function of an AR($p$) process, that of an MA($q$) process
does not vanish for large lags.  It is however bounded in absolute value by a geometrically decreasing function.
\end{quote}

The purpose of this paper is to give a clear statement on the decay rate and its short proof.
Because of the duality between AR models and MA models, it is intuitively obvious that 
the partial autocorrelation coefficients of an ARMA model decay to zero at the rate of the reciprocal of the smallest absolute value of the roots of the characteristic polynomial of the MA part of the model.  
Note that this rate is also the decay rate of the coefficients of the AR($\infty$) representation.

In literature the  sharpest results on asymptotic behavior of partial autocorrelation functions
have been given by Akihiko Inoue and his collaborators (e.g.\ \cite{inoue-2002}, \cite{inoue-kasahara-2006}, \cite{inoue-2008}, \cite{bingham-etal-2012}).  They give detailed and deep results on
the polynomial decay rate of the partial autocorrelation coefficients for  the case of long-memory processes.
Concerning ARMA processes, the most clear result seems to have been given by Inoue in Section 7 of 
\cite{inoue-2008}.  However his result is one-sided, giving an upper bound on the exponential rate, whereas Theorem \ref{th:main} in this paper gives an equality.

\section{Main result and its proof}
We consider  a zero-mean causal and invertible 
weakly stationary process $\{X_t\}_{t\in {\mathbb Z}}$ having an  
AR($\infty$) representation and an MA($\infty$) representation given by
\begin{align}
\label{eq:ar-infty-representation}
X_t&=\pi_1 X_{t-1} +\pi_2 X_{t-2} + \dots + \epsilon_t, 
\quad \pi(B)X_t = \epsilon_t, \qquad
\sum_{i=1}^\infty |\pi_i| < \infty, 
\\
\label{eq:ma-infty-representation}
X_t&= \epsilon_t + \psi_1 \epsilon_{t-1} + \dots  = \psi(B)\epsilon_t, \qquad
\sum_{i=1}^\infty |\psi_i| < \infty.
\end{align}
For an ARMA($p,q$) process 
\[
\phi(B)X_t = \theta(B)\epsilon_t,
\]
$\pi_1, \pi_2,\dots,$ decay exponentially 
with the rate of the reciprocal of the smallest absolute value of the roots 
of $\theta(B)=0$ and similarly $\psi_1, \psi_2, \dots,$
decay, with $\theta(B)$ replaced by  $\phi(B)$.
%
The autocovariance function of $\{X_t\}$ is
\[
E(X_t X_{t+k})=\gamma_k = \gamma_{-k}=\sigma_\epsilon^2 (\psi_k + \sum_{i=1}^\infty \psi_{k+i}\psi_i), \quad k\ge 0,
\]
where $\sigma_\epsilon^2 = E(\epsilon_t^2)$.
Let $H$ denote the Hilbert space spanned by $\{ X_t\}$ and for a subset $I\subset {\mathbb Z}$ of integers, let $P_I$ denote the orthogonal projector onto the
subspace $H_I$ spanned by $\{ X_t\}_{t\in I}$. The $k$-th partial autocorrelation is defined 
by  $\phi_{kk}$ in
\[
P_{[t-k,t-1]} X_t = \phi_{k1} X_{t-1} + \dots + \phi_{kk} X_{t-k}.
\]

We state our theorem, which shows that the radius of convergence is common for
the infinite series with coefficients $\{\pi_n\}$ and coefficients $\{\phi_{nn}\}$.

\begin{theorem}
\label{th:main}
Let $\{X_t\}_{t\in {\mathbb Z}}$ be a zero-mean causal and invertible 
weakly stationary process with its
AR($\infty$) representation given by \eqref{eq:ar-infty-representation} and let $\phi_{nn}$ 
be the $n$-th partial autocorrelation coefficient.  Then
\begin{equation}
\label{eq:main-result}
\limsup_{n\rightarrow\infty} |\phi_{nn}|^{1/n}=
\limsup_{n\rightarrow\infty} |\pi_n|^{1/n} .
\end{equation}
\end{theorem}

By our assumptions, both  $\{\pi_n\}$ and $\{\phi_{nn}\}$ are bounded and 
hence we have $\limsup_{n\rightarrow\infty} |\phi_{nn}|^{1/n}\le 1$, 
$\limsup_{n\rightarrow\infty} |\pi_n|^{1/n}\le 1$.  
Note that \eqref{eq:main-result} only gives the exponential decay rates of $\pi_n$
and $\phi_{nn}$  and 
does not distinguish polynomial rates
since $(n^k)^{1/n} \rightarrow 1$ as $n\rightarrow\infty$ for any power of $n$.
Akihiko Inoue and his collaborators 
provided  detailed analyses of
the polynomial decay rate of $\phi_{nn}$  for 
the case of long-memory processes (e.g.\ \cite{inoue-2002}, \cite{inoue-kasahara-2006}, \cite{inoue-2008}, \cite{bingham-etal-2012}).

For proving  Theorem \ref{th:main} we present two  lemmas.

\begin{lemma}
\label{lem:ar-to-pacf}
Suppose $\limsup_{n\rightarrow\infty} |\pi_n|^{1/n} < 1$. Then
$\limsup_{n\rightarrow\infty} |\phi_{nn}|^{1/n}\le 
\limsup_{n\rightarrow\infty} |\pi_n|^{1/n}$.
\end{lemma}

\begin{proof}
Let $\limsup_{n\rightarrow\infty} |\pi_n|^{1/n}=c_0 < 1$.
Then for every $c\in (c_0,1)$, there exist 
$n_0$ such that
\[
|\pi_n| < c^n, \qquad \forall n\ge n_0.
\]

We denote the  $h$-period ($h\ge 1$) ahead prediction by
\[
P_{[t-k,t-1]}X_{t+h-1}=\phi^{(h)}_{k1} X_{t-1} + \dots + \phi^{(h)}_{kk} X_{t-k} \qquad
(\phi^{(1)}_{kj}=\phi_{kj}).
\]
Here $\phi^{(h)}_{k1}$ is the partial regression coefficient of $X_{t-1}$ in regressing
$X_{t+h-1}$ to $X_{t-1},\dots, X_{t-k}$.  Hence it is written as
\[
\phi^{(h)}_{k1} = \frac{{\rm Cov}(P_{[t-k,t-2]}^\perp X_{t+h-1}, P_{[t-k,t-2]}^\perp 
X_{t-1})}{{\rm Var}(P_{[t-k,t-2]}^\perp X_{t-1})},
\]
where $P_{[t-k,t-2]}^\perp$ is the projector onto the orthogonal complement 
of $H_{[t-k,t-2]}$. 
Then
$|\phi^{(h)}_{k1}|$ is uniformly bounded from above as
\begin{align}
|\phi^{(h)}_{k1}| 
&\le \frac{\sqrt{{\rm Var}(P_{[t-k,t-2]}^\perp X_{t+h-1}) {\rm Var}(P_{[t-k,t-2]}^\perp X_{t-1})}}{{\rm Var}(P_{[t-k,t-2]}^\perp X_{t-1})}
\nonumber \\
&= \sqrt{\frac{{\rm Var}(P_{[t-k,t-2]}^\perp X_{t+h-1})}{{\rm Var}(P_{[t-k,t-2]}^\perp X_{t-1})}}\nonumber\\
&\le \sqrt{\frac{{\rm Var}(X_{t+h-1})}{{{\rm Var}(P_{(-\infty,t-2]}^\perp X_{t-1})}}}
= \sqrt{\frac{\gamma_0}{\sigma_\epsilon^2}}.
\label{eq:phi-h-bounded}
\end{align}

In \eqref{eq:ar-infty-representation} we apply $P_{[t-k,t-1]}$ to $X_t$. Then
\begin{align}
\phi_{k1} X_{t-1} + \dots + \phi_{kk} X_{t-k} & = P_{[t-k,t-1]} X_t \nonumber \\
&=  P_{[t-k,t-1]} P_{(-\infty,t-1]} X_t \nonumber \\
&= P_{[t-k,t-1]} (\sum_{l=1}^\infty \pi_l X_{t-l})\nonumber \\
&= \pi_1 X_{t-1} + \dots + \pi_k X_{t-k}
+\sum_{l=k+1}^\infty \pi_l P_{[t-k,t-1]} X_{t-l} .
\label{eq:projection-2}
\end{align}
Now by time reversibility of the covariance structure of weakly stationary processes
we have
\[
P_{[t-k,t-1]} X_{t-k-h}=\phi^{(h)}_{k1} X_{t-k} + \dots + \phi^{(h)}_{kk} X_{t-1}.
\]
By substituting this into \eqref{eq:projection-2}
and considering the coefficient of $X_{t-k}$ we have
\[
\phi_{kk} = \pi_k + \sum_{h=1}^\infty \pi_{k+h}\phi^{(h)}_{k1} ,
\]
where the right-hand side converges absolutely under our assumptions.
Then
\[
|\phi_{kk}| \le  |\pi_k| + \sum_{h=1}^\infty |\pi_{k+h}||\phi^{(h)}_{k1}|.
\]

For $k\ge n_0$, in view of \eqref{eq:phi-h-bounded}, 
the right-hand side is bounded as
\[
|\phi_{kk}| \le c^k  (1  + \sum_{h=1}^\infty c^h \sqrt{\gamma_0/\sigma_\epsilon^2}) =
c^k  (1+ \frac{c \sqrt{\gamma_0/\sigma_\epsilon^2}}{1-c}).
\]
Then 
\[
\limsup_{k\rightarrow\infty} |\phi_{kk}|^{1/k} \le c.
\]

Since $c>c_0$ was arbitrary, we let $c\downarrow c_0$ and obtain
\[
\limsup_{n\rightarrow\infty} |\phi_{nn}|^{1/n} \le c_0 
= \limsup_{n\rightarrow\infty}|\pi_n|^{1/n}.
\]
\end{proof}

\begin{lemma}
\label{lem:pacf-to-ar}
Suppose  $\limsup_{n\rightarrow\infty} |\phi_{nn}|^{1/n} < 1$.
Then
$
\limsup_{n\rightarrow\infty} |\pi_n|^{1/n}
\le 
\limsup_{n\rightarrow\infty} |\phi_{nn}|^{1/n}.
$
\end{lemma}

\begin{proof}
This follows from the Durbin-Levinson algorithm.  Consider $j=n$ in
\begin{equation}
\label{eq:d-l}
\phi_{n+1,j}=\phi_{n,j} - \phi_{n+1,n+1} \phi_{n,n-j+1}, \qquad j=1,2,\dots,n.
\end{equation}
The initial value is 
\[
\phi_{n+1,n}=\phi_{n,n} - \phi_{n+1,n+1} \phi_{n,1}.
\]
Using \eqref{eq:d-l} for $n$ replaced by $n+1$, $j=n$, and substituting the initial value, we obtain
\begin{align*}
\phi_{n+2,n}&=\phi_{n+1,n} - \phi_{n+2,n+2} \phi_{n+1,2} \\
&= \phi_{n,n} - \phi_{n+1,n+1} \phi_{n,1} - \phi_{n+2,n+2} \phi_{n+1,2}.
\end{align*} 
Repeating the substitution, we have
\[
\phi_{n+h,n}=\phi_{n,n} - \phi_{n+1,n+1} \phi_{n,1} - \dots
-\phi_{n+h,n+h} \phi_{n+h-1,h}.
\]
As $h\rightarrow\infty$, the left-hand side converges to $\pi_n$  (cf.\ Theorem 7.14 of \cite{pourahmadi}).
Hence
\[
\pi_n = \phi_{n,n} - \sum_{h=1}^\infty \phi_{n+h,n+h} \phi_{n+h-1,h}
\]
and 
\[
|\pi_n| \le  |\phi_{n,n}| +  \sum_{h=1}^\infty |\phi_{n+h,n+h}| |\phi_{n+h-1,h}|
\]

Now arguing as in \eqref{eq:phi-h-bounded}, we see that $|\phi_{n+h-1,h}|$ is uniformly bounded as
\[
|\phi_{n+h-1,h}| \le \sqrt{\frac{\gamma_0}{{\rm Var}(P^\perp_{(-\infty,t-1]\cup [t+1,\infty)}X_t)}}=\sqrt{\frac{\gamma_0}{{\rm Var}(P^\perp_{(-\infty,-1]\cup [1,\infty)}X_0)}} .
\]
Here the denominator is positive, because under our assumptions $\{X_t\}$ is ``minimal''
(cf.\ Theorem 8.11 of \cite{pourahmadi}, \cite[Section 2]{jewell-1983}).
The rest of the proof is the same as in the proof of Lemma \ref{lem:ar-to-pacf}
\end{proof}

By the above two lemmas, Theorem \ref{th:main} is proved as follows.

\begin{proof}[Proof of Theorem \ref{th:main}]
As noted after Theorem \ref{th:main}, both 
$\limsup_{n\rightarrow\infty} |\phi_{nn}|^{1/n}$ and
$\limsup_{n\rightarrow\infty} |\pi_n|^{1/n}$ are less than or equal to 1.  If
\[
\limsup_{n\rightarrow\infty} |\phi_{nn}|^{1/n}<  1  \ \ \text{or }\  \ 
\limsup_{n\rightarrow\infty} |\pi_n|^{1/n} < 1, 
\]
then by the above two lemmas both of them have to be less than one and 
they have to be equal.
The only remaining case is that they are equal to 1.
\end{proof}


\noindent
{\bf Acknowledgment}  \quad
This research is supported by JSPS Grant-in-Aid for Scientific Research No. 25220001.
The author is grateful for useful comments by Tomonari Sei.

\bibliographystyle{abbrv}
\bibliography{pacf}

\begin{thebibliography}{1}

\bibitem{bingham-etal-2012}
N.~H. Bingham, A.~Inoue, and Y.~Kasahara.
\newblock An explicit representation of {V}erblunsky coefficients.
\newblock {\em Statist. Probab. Lett.}, 82(2):403--410, 2012.

\bibitem{box-jenkins-5th}
G.~E.~P. Box, G.~M. Jenkins, G.~C. Reinsel, and G.~M. Ljung.
\newblock {\em Time Series Analysis: Forecasting and Control}.
\newblock Wiley Series in Probability and Statistics. John Wiley \& Sons, Inc.,
  Hoboken, NJ, fifth edition, 2015.

\bibitem{brockwell-davis}
P.~J. Brockwell and R.~A. Davis.
\newblock {\em Time Series: Theory and Methods}.
\newblock Springer Series in Statistics. Springer-Verlag, New York, second
  edition, 1991.

\bibitem{inoue-2002}
A.~Inoue.
\newblock Asymptotic behavior for partial autocorrelation functions of
  fractional {ARIMA} processes.
\newblock {\em Ann. Appl. Probab.}, 12(4):1471--1491, 2002.

\bibitem{inoue-2008}
A.~Inoue.
\newblock A{R} and {MA} representation of partial autocorrelation functions,
  with applications.
\newblock {\em Probab. Theory Related Fields}, 140(3-4):523--551, 2008.

\bibitem{inoue-kasahara-2006}
A.~Inoue and Y.~Kasahara.
\newblock Explicit representation of finite predictor coefficients and its
  applications.
\newblock {\em Ann. Statist.}, 34(2):973--993, 2006.

\bibitem{jewell-1983}
N.~P. Jewell and P.~Bloomfield.
\newblock Canonical correlations of past and future for time series:
  definitions and theory.
\newblock {\em Ann. Statist.}, 11(3):837--847, 1983.

\bibitem{pourahmadi}
M.~Pourahmadi.
\newblock {\em Foundations of Time Series Analysis and Prediction Theory}.
\newblock Wiley Series in Probability and Statistics: Applied Probability and
  Statistics. Wiley-Interscience, New York, 2001.

\end{thebibliography}

\end{document}